\documentclass[11pt, reqno,amscd]{amsart}

\usepackage{amsthm, amsmath}
\usepackage{amssymb} 
\usepackage{amscd}
\usepackage{caption}
\usepackage{cancel}
\usepackage{diagbox}

\usepackage{verbatim}
\usepackage[all]{xy}
\usepackage{rotating, graphicx, blindtext}
\usepackage{graphicx}

\usepackage{graphicx}
\usepackage{amsfonts}
\usepackage{mathrsfs}
\usepackage{color}
\usepackage{soul}
\usepackage{hyperref}

\usepackage{todonotes}

\theoremstyle{plain}
\newtheorem{theorem}{Theorem}[section]

\newtheorem{corollary}[theorem]{Corollary}

\newtheorem{definition}[theorem]{Definition}

\newtheorem{question}[theorem]{Question}

\newtheorem{lemma}[theorem]{Lemma}

\newtheorem{remark}[theorem]{Remark}

\newcommand{\Qset}{\mathbb{Q}}

\newcommand{\quash}[1]{}

\def\arrowdown#1#2{\Big\downarrow \rlap{$\vcenter{\hbox{$\scriptstyle#2$}}$}
{\hbox to -10pt{\hss{$\vcenter{\hbox{$\scriptstyle#1$}}$}}}}
\def\arrowup#1#2{\Big\uparrow \rlap{$\vcenter{\hbox{$\scriptstyle#2$}}$}
{\hbox to -10pt{\hss{$\vcenter{\hbox{$\scriptstyle#1$}}$}}}}

\numberwithin{equation}{section}

\usepackage{marginnote}

\newcommand{\calA}{\mathcal{A}}

\newcommand{\usePsi}[1]{} 

\renewcommand{\digamma}{\Psi}

\usepackage{algorithmic}
\usepackage[linesnumbered,ruled,vlined]{algorithm2e}
\usepackage{tikz}
\usepackage{tikz-cd}
\usepackage{tikz-qtree}
\usepackage{tcolorbox}

\newtheorem{hypothesis}[theorem]{Hypothesis}

\usetikzlibrary{decorations.pathreplacing}

\usepackage{tikz-cd}

\begin{document}
\title [On the radii of Voronoi cells of rings of integers]
{On the radii of Voronoi cells of rings of integers }
\vspace*{-1cm}

\date{\today}

\author[F.M. Bleher]{Frauke M. Bleher}
\address{F.B.: University of Iowa, Iowa City, IA 52242, USA}
\email{frauke-bleher@uiowa.edu}
\thanks{The first author was supported in part by Simons Foundation grant No. 960170 and NSF FRG grant DMS-2411703.}

\author[T. Chinburg]{Ted Chinburg}
\address{T.C.: University of Pennsylvania, Philadelphia, PA 19104, USA}
\email{ted@math.upenn.edu}
\thanks{The second author was supported in part by Simons Foundation grant No. MP-TSM-00002279 and  NSF FRG grant DMS-2411702.}

\author[X. Ding]{Xuxi Ding}
\address{X.D.: University of Pennsylvania, Philadelphia, PA 19104, USA}
\email{xuxiding@math.upenn.edu}

\author[N. Heninger]{Nadia Heninger}
\address{N.H.: University of California at San Diego, San Diego, CA 92093}
\email{nadiah@cs.ucsd.edu}
\thanks{The fourth author was supported in part by NSF FRG grant DMS-2411704.}

\author[D. Micciancio]{Daniele Micciancio}
\address{D.M.: University of California at San Diego, San Diego, CA 92093}
\email{dmicciancio@cs.ucsd.edu}
\thanks{The fifth author was supported in part by NSF FRG grant DMS-2411704.}

\begin{abstract}
Since the time of Minkowski a basic problem in number theory has been to find lower bounds for the absolute value $\Delta(K)$ of the discriminant of a number field $K$ in terms of the degree $n(K)$ of $K$.  In this paper we study another measure of the size of $K$ given by the covering radius $\mu(K)$ of the ring of integers $O_K$ of $K$.  Here $\mu(K)$ is the $L^2$ radius $||V_2(K)||_2$ of the $L^2$ Voronoi cell $V_2(K)$ of  $O_K$, where $V_2(K)$ is the set of points in $\mathbb{R} \otimes_{\mathbb{Q}} K$ that are at least as close to the origin as they are to any non-zero element of $O_K$.  To put a limit on what lower bounds one can prove for $\mu(K)$ in terms of $n(K)$, we study infinite families of $K$ of increasing degree for which $\mu(K)$ can be bounded above by an explicit power of $n(K)$.   We also study analogous questions when the $L^2$ norm is replaced by the $L^\infty$ norm.
\end{abstract}

\maketitle

\section{Introduction}

Two measures of the size of a number field $K$ are the absolute value $\Delta(K)$ of its discriminant and the degree $n(K)$ of $K$ over $\mathbb{Q}$. This paper is about a third measure of the size of $K$ given by the covering radius $\mu(K)$ of the ring of integers $O_K$ of $K$.  Here $\mu(K)$ is the radius $||V_2(K)||_2$ in the $L^2$ metric
on $\mathbb{R} \otimes_{\mathbb{Q}} K$ of the Voronoi cell $V_2(K)$ of $O_K$.  The $L^2$ metric on $\mathbb{R} \otimes_{\mathbb{Q}} K$ arises from the real and complex embeddings of $K$.  The Voronoi cell $V_2(K)$ is the set of elements of
$\mathbb{R} \otimes_{\mathbb{Q}} K$ whose distance to the origin is less than or equal to its distance to any non-zero element of the integers $O_K$ of $K$.
The study of discriminants amounts to studying the volume $\mathrm{Vol}(V_2(K))$ of $V_2(K)$.  This is because $V_2(K)$ is the closure of a fundamental domain for the translation action of $O_K$ on $\mathbb{R} \otimes_{\mathbb{Q}} K$, so 
\begin{equation}
\label{eq:delrel}
\Delta(K) = 2^{2 r_2(K)} \mathrm{Vol}(V_2(K))^2
\end{equation}
 where $2r_2(K) \le n(K)$ is the number of non-real complex embeddings of $K$.  

Much of the research on discriminants since the work of Minkowski has focused on finding lower bounds for $\Delta(K)$ in terms of the degree $n(K)$.  A famous result of Golod and Shafarevitch \cite{GS}  is that there is a constant $c > 1$ such that there are $K$ of arbitrarily large degree $n(K)$ such that the root discriminant $\delta(K) = \Delta(K)^{1/n(K)}$ is bounded above by $c$.    Let $c_0$ be the infimum of all $c$ for which this is true.  There is a large literature on lower bounds for $c_0$, beginning with techniques from the geometry of numbers and followed by analytic methods that exploit the appearance of $\Delta(K)$ in the functional equation of the zeta function of $K$ (see \cite{O}, \cite{Bar}, \cite{KP1} and their references).   In view of (\ref{eq:delrel}), $c_0$ is the infimum of all $c > 0$ such that there  are number fields $K$ of arbitrarily large degree for which $\mathrm{Vol}(V_2(K)) \le 2^{-r_2(K)} c^{n(K)/2}$.

The primary difference in studying lower bounds for $||V_2(K)||_2$ rather than for $\delta(K)$ as a function of $n(K)$  is that one needs additional control on the ring of integers $O_K$ in order to analyze $V_2(K)$.  The following is a very strong question about $||V_2(K)||_2$ inspired by the above result of Golod and Shafarevitch:

\begin{question}
\label{q:firstquest1}
Is there a constant $d > 0$ such that there are number fields $K$ of arbitrarily large degree $n(K)$ such that $||V_2(K)||_2 \le d \sqrt{n(K)}$?    If so, what is the infimum  $d_2$ of the set of all such $d$?
\end{question}

The choice of $d \sqrt{n(K)}$ here arises from the fact that such an upper bound is the minimal one that is consistent with the existence of the constant $c_0$ above.  To explain this, if $||V_2(K)||_2 \le d \sqrt{n(K)}$ then $V_2(K)$ is contained in a ball of radius $d \sqrt{n(K)}$ and has $L^2$ volume bounded by that of this ball.  On the other hand, the constant $c_0$ above is connected with whether there are fields of arbitrarily large degree having $V_2(K)$ of very small volume.  A calculation of the volume of  the $L^2$ ball of radius $d \sqrt{n(K)}$ shows that in fact, if there is a  $d$  as in Question \ref{q:firstquest1}, then there is a lower bound for $d$ in terms of $c_0$.   (See the proof of Theorem \ref{thm:volbound} below.)  However, a proof that there is a constant $d$ as in Question \ref{q:firstquest1} seems beyond present methods.  So instead we ask the following question.

\begin{question}
\label{q:secondeq}
What is the infimum  $\nu_2$ of all  constants $\nu > 0$ such that there exist number fields 
$K$ of arbitrarily large degree $n(K)$ for which $||V_2(K)||_2 \le n(K)^{\nu}$?
\end{question}

Note that if there is any $d$ as in Question \ref{q:firstquest1}, then $\nu_2 = 1/2$.  We can now state our main result:

\begin{theorem}
\label{thm:tauthm1}  The constant $\nu_2$ is well defined and satisfies
\begin{equation}
\label{eq:taubound} 
1/2 \le \nu_2 \le \frac{1}{2}+ \mathrm{log}(\sqrt{5}/2)/\mathrm{log}(2) = 0.6609...
\end{equation}
\end{theorem}

The lower bound in this theorem is a consequence of the arithmetic geometric mean inequality together with the fact that norms of elements of $O_K$ lie in $\mathbb{Z}$.  
To show the upper bound,  we must construct a family of number fields of increasing degree for which the size of $||V_2(K)||_2$ can be suitably controlled as a function of $n(K)$.  

A natural family of fields to consider would be cyclotomic fields.   However, we will show in Corollary \ref{cor:cyclocor} that families of cyclotomic fields cannot lead to upper bounds on $\nu_2$ in Question \ref{q:secondeq} that are smaller than $1$.  Instead, we will consider  families of fields $N$ of small root discriminant that were constructed in \cite{BC}.  These $N$ were originally constructed to solve a problem of Peikert and Rosen in \cite{PR} concerning the efficient construction of infinite families of fields of increasing degree and small root discriminant as a function of the degree.  

The $N$ we consider are the top of  two-step towers $\mathbb{Q} \subset L \subset N$ of elementary abelian two-extensions $N/L$ and $L/\mathbb{Q}$.   The top step $N/L$ in each of these towers is unramified at all finite primes.  Equivalently, the root discriminants of the two fields $N$ and $L$  forming the top step are equal.  One byproduct of our work is a description in Theorem \ref{thm:biggest} below of the square roots one must adjoin to an arbitrary number field  to construct the maximal elementary abelian two-extension of that field that is unramified over all finite primes.  

We now return to considering an arbitrary number field $K$.   In addition to the $L^2$ norm on $\mathbb{R} \otimes_{\mathbb{Q}} K$, one can consider the $L^\infty$ norm which arises as the sup of the Euclidean absolute values of the images of elements under the algebra homomorphisms to $\mathbb{C}$ induced by the real and complex embeddings of $K$.  One can define an $L^\infty$ Voronoi cell $V_\infty(K)$ to be the set of elements of $\mathbb{R} \otimes_{\mathbb{Q}} K$ that are at least as close in the $L^\infty$ norm to $0$ as to any non-zero element of $O_K$.
Our methods lead to the following result concerning the $L^\infty $ radius $||V_\infty(K)||_\infty$ of $V_\infty(K)$.  

\begin{theorem}
\label{thm:tauinfthm1}  Let  $\nu_\infty$ be the infimum of all  constants $\nu > 0$ such that there exist number fields 
$K$ of arbitrarily large degree $n(K)$ for which $||V_\infty(K)||_\infty \le n(K)^{\nu}$.
Then 
\begin{equation}
\label{eq:tauinfbound} 
0 \le \nu_\infty \le \mathrm{log}(3/2)/\mathrm{log}(2) = 0.5849...
\end{equation}
\end{theorem}

Our approach to proving Theorems \ref{thm:tauthm1} and \ref{thm:tauinfthm1} uses fundamental domains for subgroups of finite index in $O_N$ when $N$ is the top of a two-step tower of the kind described above.  
  Let $\epsilon > 0$ be a real constant.  We will construct a $\mathbb{Z}$-basis $\{b_{i,\epsilon}\}_{i = 1}^{n(N)}$ for a subgroup $O_{N,\epsilon}$ of finite index in $O_N$ along with an explicit fundamental domain $\mathcal{F}_{N,\epsilon}$ for the translation action of $O_{N,\epsilon}$ on $\mathbb{R} \otimes_{\mathbb{Q}} N$.   This $\mathcal{F}_{N,\epsilon}$ then contains a fundamental domain for $O_N$.  Theorems \ref{thm:tauthm1} and \ref{thm:tauinfthm1} result from proving upper bounds  on the $L^2$ and $L^\infty$ radii of $\mathcal{F}_{N,\epsilon}$.  The reason for introducing $\epsilon$ into this construction is that as $\epsilon \to 0^+$ these bounds improve.

 We now outline the contents of this paper.
 
In \S \ref{s:fields} we study how to optimize the construction in \cite{BC} of two-step towers of number fields in which each step is an elementary abelian two-extension and the top step is unramified at all finite primes.    
This leads to some open problems concerning units congruent to $1$ mod $4$ and the size of the two-torsion in narrow ideal class groups.   
 
In \S \ref{s:fundament} we prove Theorems \ref{thm:tauthm1} and \ref{thm:tauinfthm1} by constructing the fundamental domains $\mathcal{F}_{N,\epsilon}$ mentioned above.  

In \S \ref{s:cyclocase} we prove a lower bound for $||V_2(K)||_2$ in terms of $\delta(K) $ for arbitrary number fields $K$.  We show that for any $\epsilon > 0$ all but finitely many cyclotomic fields $K$ have $||V_2(K)||_2$ and $\Delta(K)^{1/n(K)}$ greater than or equal to $n(K)^{1 - \epsilon}$. 

\medskip

\textbf{Acknowledgment.} The authors would like to thank Adam Suhl for many useful comments and for checking some of the computations in this paper.

\section{Towers of elementary abelian two-extensions of number fields}
\label{s:fields}
\newif\ifmathcompetitive

The goal of \cite{BC} was to give an efficient construction of an infinite family of number fields $N$ of increasing degree $n(N)$ such that for all $\epsilon > 0$,
the root discriminant $\delta(N) = \Delta(N)^{1/n(N)}$
is bounded by $n(N)^\epsilon$ for all but finitely many $N$ in the family.
The $N$ involved are the top of two-step towers $\Qset\subset L \subset N$ of elementary abelian two-extensions $L/\Qset$ and $N/L$ in which $N/L$ is unramified at all finite primes.  By the multiplicativity of differents in towers, 
the last condition  is equivalent to $\delta(N) = \delta(L)$.

The strategy used in \cite{BC} for constructing $N$ as above for which $\delta(N) \le n(N)^\epsilon$ was to find $L$ for which 
$[N:L] = n(N)/n(L)$ is very large in comparison to $\delta(L)$.  Then 
$$\delta(N) = \delta(L) \quad \mathrm{and}\quad n(N)^\epsilon = n(L)^\epsilon \cdot [N:L]^\epsilon$$
will imply $\delta(N) \le n(N)^\epsilon$ if $\delta(L) \le [N:L]^\epsilon$.

In this section we will be concerned with analyzing the above construction.  

 \subsection{Elementary abelian two-extensions of number fields.}
 \label{s:unramelem}

We begin with a characterization of quadratic extensions for which the root discriminant does not change.   
  
\begin{theorem}\label{thm:main}
\label{thm:quadratic} Let $L$ be a number field and suppose $w \in L^*$.  The following are equivalent:
\begin{enumerate}
\item[i.] The extension $L' = L(\sqrt{w})$ is unramified over all prime ideals of $O_L$.
\item[ii.] The fields $L'$ and $L$ have the same root discriminant.
\item[iii.] There is an element $\beta \in L^*$ such that $\beta^2 w \in 1 + 4 O_L$ and $\beta^2 w O_L$  is the square of an ideal $\mathcal{A}$  of $O_L$.  
\end{enumerate}
Under these conditions, the integers $O_{L'}$ of $L'$ are generated as an $O_L$-module by  $1$, $\frac{1 +  \beta \sqrt{w}}{2}$ and $\mathcal{A}^{-1}  \beta \sqrt{w}$, where $\mathcal{A}^{-1}$ is the inverse ideal to $\mathcal{A}$.
\end{theorem}
\begin{proof}
The equivalence of (i) and (ii) follows from the multiplicativity of differents and the fact that (i) is true if and only if the relative discriminant ideal of $O_{L'}$ over $O_L$ is the ideal $O_L$.  We now show (i) is equivalent to (iii).

Suppose first that $L' = L(\sqrt{w})$ is unramified over all prime ideals $\mathcal{P}$ of  $O_L$.  If $w$ has odd valuation at $\mathcal{P}$ then $L'$ is ramified over $\mathcal{P}$, contrary to hypothesis.  Hence $w O_L$ is the square of an ideal of $O_L$.  The existence of a $\beta' \in L^*$ such that $\beta'^2 w$ generates the square of an ideal that is prime to $2O_L$ follows from the strong approximation theorem. 

  We now show that for each prime ideal $\mathcal{P}$ over $2$ in $O_L$, there is a $\gamma_{\mathcal{P}} \in O_L$ such that 
$\mathcal{P}$ is inert in the quadratic extension $L(\sqrt{1 + 4\gamma_{\mathcal{P}}})$ of $L$.  Pick any $\gamma_{\mathcal{P}} \in O_L$ whose image in the residue field $k(\mathcal{P}) = O_L/\mathcal{P}$ is not of the form $y^2 - y$ for some $y \in k(\mathcal{P})$.  Then $z = (1 + \sqrt{1 + 4\gamma_{\mathcal{P}}})/2$ is a root of the equation $z^2 - z - \gamma_{\mathcal{P}} = 0$, so $z$ is an algebraic integer.  This equation has no roots in the residue field $k(\mathcal{P}) = O_L/\mathcal{P}$ by our choice of $\gamma_{\mathcal{P}}$, so it is irreducible and $L(z) = L(\sqrt{1 + 4\gamma_{\mathcal{P}}})$ is a quadratic extension of $L$ that is inert over $\mathcal{P}$.  
Since $L' = L(\sqrt{\beta'^2 w})$ and $L'/L$ has been assumed to be unramified, we know by Kummer theory over the completion $L_{\mathcal{P}}$ that there is a $\xi_{\mathcal{P}} \in L_{\mathcal{P}}^*$ such that either $\xi_{\mathcal{P}}^2 \beta'^2 w = 1$ or $\xi_{\mathcal{P}}^2 \beta'^2 w = 1 + 4 \gamma_{\mathcal{P}}$.  Here $\xi_{\mathcal{P}}$ has valuation $0$ at $\mathcal{P}$.  Let $O_{\mathcal{P}}$ be the valuation ring of $L_{\mathcal{P}}$.  Applying the strong approximation theorem again, we find there is a $\xi \in O_L$ that is congruent to $\xi_{\mathcal{P}}$ mod $4 O_{\mathcal{P}}$ for every prime ideal $\mathcal{P}$ over $2$.  When we now define 
$\beta = \xi \beta'$ we get that $\beta^2 w \in 1 + 4 O_L$ generates the square of an ideal that is prime to $2O_L$, so (iii) holds.

Conversely, suppose (iii) holds.  Condition (iii) implies that $L' = L(\sqrt{w}) = L(\sqrt{\beta^2 w})$ is unramified over all prime ideals of $O_L$ of odd residue characteristic and  $\beta^2 w = 1 + 4\tau$ for some $\tau \in O_L$.  It will now suffice to show $L(\sqrt{1 + 4\tau})$ is unramified over each prime ideal $\mathcal{P}$ of $O_L$ that divides $2O_L$.  By Kummer theory over
$L_{\mathcal{P}}$, it will suffice to show that 
\begin{equation}
\label{eq:inclus}
1 + 4O_{\mathcal{P}} = (1 + 2O_{\mathcal{P}})^2 \cup (1 + 2O_{\mathcal{P}})^2 \cdot (1+ 4 \gamma_{\mathcal{P}})
\end{equation}
 when $\gamma_{\mathcal{P}}$ is constructed as in the first part of the proof.  
 
 Let $\pi_{\mathcal{P}}$ be a uniformizer in $O_{\mathcal{P}}$.
For $v \in O_{\mathcal{P}}$ we have 
\begin{equation}
\label{eq:veq}(1 + 2v)^2 = 1 + 4(v + v^2)\equiv 1 + 4(v^2 - v)\quad \mathrm{ mod }\quad 4\pi_{\mathcal{P}} O_{\mathcal{P}}.
\end{equation}  There is an isomorphism of groups
$$\frac{(1+ 4 O_{\mathcal{P}})^* }{(1 + 4 \pi_{\mathcal{P}} O_{\mathcal{P}})^* }\to k(\mathcal{P})^+$$
that sends $1 + 4 v $ to the residue class of $v$.   Since $\gamma_\mathcal{P}$ represents the non-trivial
coset of $ \{y^2 - y: y \in k(\mathcal{P})\} $ in $k(\mathcal{P})$, we conclude from (\ref{eq:veq}) that 
$$(1 + 2O_{\mathcal{P}})^2 \cup (1 + 2O_{\mathcal{P}})^2 \cdot (1+ 4 \gamma_{\mathcal{P}})$$
lies in $1 + 4 O_{\mathcal{P}}$ and surjects onto $(1 + 4 O_{\mathcal{P}})^* / (1 + 4 \pi_{\mathcal{P}} O_{\mathcal{P}})^*$.  Hence
it will now be enough to show that 
$$1 + 4 \pi_{\mathcal{P}} O_{\mathcal{P}} \subset (1+ 2 O_{\mathcal{P}})^2.$$
For this it suffices to show that for all $i \ge 1$ one has
$$1 + 4 \pi_{\mathcal{P}}^i O_{\mathcal{P}} \subset (1 + 2 \pi_{\mathcal{P}}^i O_{\mathcal{P}})^2\cdot  (1 + 4 \pi_{\mathcal{P}}^{i+1} O_{\mathcal{P}}).$$
This is clear from 
$$(1 + 2\pi_{\mathcal{P}}^i y)^2 = 1 + 4 \pi_{\mathcal{P}}^i  y + 4 \pi_{\mathcal{P}}^{2i} y^2.$$

This completes the proof that (i) and (iii) are equivalent, so (i), (ii) and (iii) are equivalent by our earlier remarks.

Suppose now that (iii) holds.  If $L= L'$ the last statement of the Lemma is clear, so suppose
$L'$ is quadratic over $L$.  The element $\rho = \frac{1 + \beta \sqrt{w}}{2}$ is a zero of the polynomial $z^2 - z+ \frac{1 - \beta^2 w}{4}$. This polynomial has coefficients in $O_L$, so $\rho$ is an algebraic integer in $L'$.    Since $\mathcal{A}^{-2} \beta^2 w  = O_L$ one has $\mathcal{A}^{-1} \beta \sqrt{w} \subset O_{L'}$.   The relative discriminant over $O_L$ of the $O_L$-submodule $O_L + O_L \mathcal{A}^{-1} \beta \sqrt{w}$ of $O_{L'}$ is $4 \mathcal{A}^{-2} \beta^2 w = 4 O_L$, so the index of this submodule in $O_{L'}$ is a power of $2$.  The discriminant of  the minimal polynomial over $L$ of $\rho = \frac{1 + \beta \sqrt{w}}{2}$ is
$$\left ( \frac{1 + \beta \sqrt{w}}{2} - \frac{1 - \beta \sqrt{w}}{2} \right )^2 = \beta^2 w$$
which generates an ideal that is prime to $2 O_L$.  Thus the index of the subgroup $O_L + O_L \rho$ in $O_{L'}$ is odd.
This implies that the $O_L$-submodule generated by $1$, $\rho$ and $\mathcal{A}^{-1} \beta \sqrt{w}$
is $O_{L'}$ as claimed. 
\end{proof}

The following corollary gives a sufficiently large set of $w$ to construct every quadratic extension of $L$ with the same root discriminant as $L$.

\begin{corollary}\label{cor:main}
  For any number field $L$, let $w \in 1 + 4 O_L$ be such that $w \cdot O_L = \calA^2$ is the square of an integral ideal $\calA$.
  Then, $L' = L(\sqrt{w})$ has the same root discriminant as $L$, and its ring of integers $O_{L'}$ is generated (as an $O_L$-module)
  by $1$, $\frac{1+\sqrt{w}}{2}$ and $\calA^{-1}\sqrt{w}$.
  Moreover, any quadratic extension of $L$ with the same root discriminant can be expressed in this way.
\end{corollary}

The next corollary describes all elementary abelian two-extensions of $L$ that have the same root discriminant as $L$.
Recall that the narrow ideal class group $\mathrm{Cl}(L)^+$ of $L$ is the quotient of the group of all fractional ideals of $L$ by the subgroup of principal ideals that have a generator that is positive under all the embeddings of $L$ into $\mathbb{R}$.

\begin{corollary}
\label{cor:multiquadconstruct}
For $i = 1,\ldots,m$ let $w_i \in L^*$ satisfy the conditions of Theorem \ref{thm:quadratic}.  Suppose no non-trivial product of a finite subset of $\{w_1,\ldots,w_m\}$  lies in $(L^*)^2$.  
\begin{enumerate}
\item[i.] The field  
$N = L(\sqrt{w_1},\ldots,\sqrt{w_m})$ is an elementary abelian two-extension of $L$ of degree $2^m$ that is unramified over all finite primes, so that its root discriminant is the same as that of $L$.  
\item[ii.] The ring of integers $O_N$ is generated by the rings of integers $O_{L(\sqrt{w_i})}$ as $i = 1,\ldots,m$, and $O_N$ is isomorphic to the tensor product over $O_L$ of these rings. 
\item[iii.] The unique maximal $N$ of this kind is associated to any choice of $W = \{w_1,\ldots,w_m\}$ with these properties for which $m$ is the two-rank of the two-torsion $\mathrm{Cl}(L)^+[2]$ in the narrow ideal class group $\mathrm{Cl}(L)^+$ of $L$.
\end{enumerate}
\end{corollary}

\begin{proof} By Kummer theory, $N$ is Galois with group $(\mathbb{Z}/2)^m$ over $L$.  The compositum of two disjoint extensions of $L$ that are unramified over all finite primes is unramified over all finite primes, and the ring of integers of such a compositum is isomorphic to the tensor product over $O_L$.  The rings of integers of the extensions.  By class field theory, the unique maximal elementary abelian two-extension of $L$ that is unramified over all finite primes has Galois group over $L$ isomorphic to the quotient $\mathrm{Cl}(L)^+/2\mathrm{Cl}(L)^+$.  This quotient has the same order as that of $\mathrm{Cl}(L)^+[2]$.
\end{proof}

The following results give a description by Kummer theory  of the maximal $N$ of Corollary \ref{cor:multiquadconstruct}.  
We have the following consequence of the strong approximation theorem.

\begin{lemma}
\label{lem:defineit} Let $L$ be an arbitrary number field.  Define $Z(L)$ to be the multiplicative subgroup of elements $\alpha \in L^*$ such that $\alpha O_L$ is the square of a fractional ideal $I_\alpha$.  Let $\mathrm{Cl}(L)[2]$ be the two-torsion of the ideal class group of $L$. There are homomorphisms
\begin{equation}
\label{eq:xipi}
\xi: Z(L) \to \mathrm{Cl}(L)[2] \quad \mathrm{and} \quad  
\pi:Z(L) \to (O_L/4O_L)^*/((O_L/4O_L)^*)^2
\end{equation} defined in
the following way.  For $\alpha \in Z(L)$ let $\xi(\alpha)$ be the ideal class of $I_\alpha$. There is an element $\beta \in L^*$ such that 
$\alpha \beta^2$ lies in $O_L$ and is invertible mod $2O_L$. Let $\pi(\alpha)$ be the image of $\alpha \beta^{2}$ in $(O_L/4O_L)^*/((O_L/4O_L)^*)^2$. 
\end{lemma}

\begin{proof}
The homorphism $\xi$ is well defined since when $\alpha, \alpha' \in Z(L)$ we have $I_{\alpha \cdot \alpha'} = I_\alpha \cdot I_{\alpha'}$.  To show $\pi$ is well defined, suppose $\alpha \in Z(L)$.  Then $\alpha O_L = I_\alpha^2$.  We can find an element $\gamma \in L^*$ such that $I_\alpha \cdot \gamma$ is an ideal of $O_L$ that is prime to $2O_L$.  Then $\alpha \gamma^2 O_L$ is a fractional ideal that is prime to $2O_L$.  We can find an element $0 \ne \tau \in O_L$ such that $\tau O_L$ is prime to $2 O_L$ and $\alpha \gamma^2 \tau^2 \in O_L$.  Now setting $\beta = \gamma \tau$ shows
$\alpha \beta^{2} \in O_L$ and $\alpha \beta^{2}$ is prime to $2O_L$.  To complete the proof that
$\pi$ is well defined, we suppose that $\beta' \in L^*$ is another element such that $\alpha (\beta')^{2} \in O_L$ is prime to $2 O_L$.  We need to show that 
$\alpha \beta^{2}$ and $\alpha (\beta')^{2}$ have the same image in $(O_L/4O_L)^*/((O_L/4O_L)^*)^2$.
This follows from the fact that 
$$\frac{\alpha \beta^{2}}{\alpha (\beta')^{2}} = \left ( \frac{\beta}{\beta'} \right )^2$$
generates a fractional ideal  that is prime to $2O_L$, so $\beta/\beta'$ also generates such an ideal and defines a class in $(O_L/4O_L)^*$.
\end{proof}

\begin{theorem}
\label{thm:biggest}
There is a finitely generated subgroup $V$  of $Z(L)$ that contains the unit group $O_L^*$ and for which $\xi(V) = \mathrm{Cl}(L)[2]$ when $\xi$ is as in (\ref{eq:xipi}).  For any such $V$ let $V' = V \cap \mathrm{Kernel}(\pi)$.   Let $W = \{w_1,\ldots,w_m\}$ be a maximal finite set of element of $V'$ whose images in $L^*/(L^*)^2$ are independent, i.e. such that no non-empty product of the elements of a subset of $W$ lies in $L^*$.  
Then $N = L(\sqrt{w_1},\ldots,\sqrt{w_m})$ is the unique maximal elementary abelian two-extension of $L$ that has the same root discriminant as $L$.  In consequence,
$2^m = \# \mathrm{Cl}(L)^+[2]$.

\end{theorem}

\begin{proof}
There is a $V$ with the stated properties because $O_L^*$ and the two-torsion $\mathrm{Cl}(L)[2]$ of the class group of $L$ are finitely generated.  In view of Corollary \ref{cor:multiquadconstruct}, to complete the proof it will suffice to show that if $N'$ is a quadratic extension of $L$ that is unramified over all prime ideals of $O_L$, then $N' = L(\sqrt{w'})$ for some $w' \in V' = V \cap \mathrm{Kernel}(\pi)$.
We know $N' = L(\sqrt{w})$ for some $w$ having the properties in Theorem \ref{thm:quadratic}. Then $w O_L =  
I^2$ for some non-zero ideal $I$ of $O_L$.  The ideal class $[I]$ of $I$ has order $1$ or $2$.  Therefore there is an element $v \in V$ such that $\xi(v) = [I]$.  Hence there is a $\gamma \in L^*$ such that $(I\gamma )^2 = v O_L$ as ideals.  Thus $w \gamma^{2} O_L = v O_L$, so $w \gamma^{2} \in v\cdot O_L^* \subset V$ because $O_L^*$ is contained in the group $V$.  Replacing $w$ by $w \gamma^2$ does not change $N' = L(\sqrt{w})$, so after making this replacement we can assume $w \in V$.   Now since $L(\sqrt{w})/L$ is unramified over all primes of $O_L$,  Theorem \ref{thm:quadratic} shows $\pi(w)$ is trivial when $\pi$ is defined as in Lemma \ref{lem:defineit}.  Thus $w \in V \cap \mathrm{Kernel}(\pi) = V'$, which completes the proof.
\end{proof}

Note that there may be many $V$ and  $W$ for which the conclusions of Theorem \ref{thm:biggest} hold, but all such pairs determine the same maximal $N$.

Recall that if $\mathcal{S}$ is a finite set of prime ideals of $O_L$, the subgroup $U(\mathcal{S}) \subset L^*$ of $\mathcal{S}$-units of $L$ is the multiplicative group of all $\alpha \in L^*$
 such that the ideal $\alpha O_L$ is a product of integral powers of elements of $\mathcal{S}$.  The group $U(\mathcal{S})$ is finitely generated. 
 
 \begin{corollary}
 \label{cor:niceroots}  Let $\mathcal{S}$ be a finite set of prime ideals of $O_L$ such that the subgroup of $\mathrm{Cl}(L)$ generated by the prime ideals in $\mathcal{S}$ contains the two-torsion of $\mathrm{Cl}(L)$.  Then  in  Theorem \ref{thm:biggest} we can let $V$ be $U(\mathcal{S}) \cap Z(L)$.  
\end{corollary}  

\begin{remark}
\label{rem:mink}
Minkowski showed that every ideal class in $\mathrm{Cl}(L)$ is represented by an integral ideal of norm bounded by 
$$M_L = \left ( \frac{4}{\pi} \right )^{r_2(L)} \frac{n(L)!}{n(L)^{n(L)}} \sqrt{\Delta(L)}.$$
Thus in Corollary \ref{cor:niceroots} we can let $\mathcal{S}$ be the set of prime ideals of $O_L$ having norm bounded by $M_L$.  
\end{remark}

The following special case of Theorem \ref{thm:biggest} includes the fields constructed in \cite{BC} and has the virtue of leading to an explicit basis for $O_N$ as a free $O_L$-module.

 \begin{theorem}
\label{ex:unitex}  Suppose $W = \{w_1,\ldots,w_m\}$ is any subset of $O_L^*$ such that each $w_i$ defines a class in $(O_L/4O_L)^*$ that is a square and no non-nontrivial product of the $w_i$ lies in $(O_L^*)^2$.  Then the field $N = L(\sqrt{w_1},\ldots,\sqrt{w_m})$  has $\delta(N) = \delta(L)$. If each of the $w_i$ has trivial image in $(O_L/4O_L)^*$ then the ring of integers $O_N$ of $N$ is the free $O_L$-module on the elements
\begin{equation}
\label{eq:etaTdef}
\eta_T = \prod_{i \in T} \frac{1 + \sqrt{w_i}}{2}
\end{equation}
where $T$ ranges over all subsets of $\{1,\ldots,m\}$.
\end{theorem}

\begin{proof}
Corollary \ref{cor:multiquadconstruct}
 shows $\delta(N) = \delta(L)$ and that  $O_N$ is the tensor product over $O_L$ of the
rings of integers of the fields $L(\sqrt{w_i})$. Let $w = w_i \in O_L^*$ in Theorem \ref{thm:quadratic}.  We can take $\beta = 1$ in  Theorem \ref{thm:quadratic},which leads to $\mathcal{A} = O_L$ in this theorem because $w \in O_L^*$.  Now Theorem \ref{thm:quadratic} shows $O_{L(\sqrt{w_i})}$ is generated over $O_L$ by $1$, $\rho = \frac{1 + \sqrt{w_i}}{2}$ and $\mathcal{A}^{-1} \beta \sqrt{w_i} = O_L \sqrt{w_i}$.  However, $\sqrt{w_i} \in O_L + O_L \rho$, so we conclude $O_{L(\sqrt{w_i})}$ is the free rank two $O_L$-module generated by $1$ and
$\rho$.  Theorem \ref{ex:unitex} is now clear from the fact that $O_N$ is the tensor product of the $O_{L(\sqrt{w_i})}$ over $O_L$.
\end{proof}

We finish this section with a discussion of the maximal size of $m$  in Corollary \ref{cor:multiquadconstruct}.  By this corollary, the maximal value of $m$ is the
 the two-rank of the narrow ideal class group $\mathrm{Cl}^+(L)$ of $L$.  Here $\mathrm{Cl}^+(L)$ is an extension of the class group $\mathrm{Cl}(L)$ by an elementary abelian $2$-group of rank bounded by the number $r_1(L)$ of real places of $L$. So 
\begin{equation}
\label{eq:mclass}
m \le \mathrm{log}_2(\# \mathrm{Cl}^+(L)) \le \mathrm{log}_2(h_L) + r_1(L) 
\le \mathrm{log}_2(h_L) + n(L)
\end{equation}
when $h_L = \# \mathrm{Cl}(L)$ is the class number of $L$. 
The Brauer-Siegel Theorem and Zimmert's lower bounds on regulators give effective upper bounds on $h_L$ in terms of the degree of $L$ and the absolute value $\Delta(L)$ of the discriminant of $L$. For example, it follows from \cite[ eq. (2)]{Queme} that for each $\epsilon > 0$, all but finitely many $L$ of a given degree $n = n(L)$ satisfy
\begin{equation}
\label{eq:upperh}
\mathrm{log}(h_L) \le (1/2 + \epsilon) \mathrm{log}(\Delta(L)).
\end{equation}
Combining (\ref{eq:mclass}) and (\ref{eq:upperh}) this gives the bound 
\begin{equation}
\label{eq:uppermtwo}
m \le (1/2 + \epsilon) \mathrm{log}_2(\Delta(L)) + n(L)
\end{equation}
for all but finitely many $L$ of a given degree.

If one uses the construction in Theorem \ref{ex:unitex}, an easy upper bound on $m$ is the two-rank of $O_L^*/(O_L^*)^2$.  By the Dirichlet unit theorem, $O^*_L/(O_L^*)^2$ has two-rank bounded by $n(L)$, so $m \le n(L)$.    A case of interest  is when 
\begin{equation}
\label{eq:niceL}
L = \mathbb{Q}(\sqrt{p_1},\ldots,\sqrt{p_\ell})
\end{equation} for a set of $\ell$ distinct primes $p_1,\ldots,p_\ell$ that are congruent to $1$ mod $4$.  Then we have $m \le n(L) = 2^\ell$, so in principle $n(N)$ could be as large as
$2^{\ell + 2^\ell}$.  Since $\delta(N) = \delta(L)$, if $m = 2^\ell$ then 
this would produce an $N$ with the same root discriminant of $L$ but having degree $n(N) = n(L) \cdot 2^{n(L)}$.

However, it seems beyond present techniques to produce an infinite family of examples for which it can be proved that $m$ is on the order of  $n(L)$.  It is shown in \cite{BC} that one can give an effective construction, in the sense recalled below, of an infinite sequence of fields $L$ of the form in  (\ref{eq:niceL})
with $\ell \to \infty$ for which $m \ge c \ell^2$ for some universal positive constant $c$.  These fields are constructed by the method in  Theorem \ref{ex:unitex}.  The resulting $N$ thus
have $n(N) \ge 2^{\ell + c \ell^2}$ and the same root discriminant as $L$.  This was shown in 
\cite{BC} to imply that for these families and every $\epsilon > 0$, one has $\delta(N) < n(N)^\epsilon$ for all but finitely many $N $ in the family.  Note here, though, that the inequalities 
$$c \ell^2 \le m \le 2^\ell$$ 
show there is a large gap between the unconditional lower bound and the naive upper bound resulting from the construction in Theorem \ref{ex:unitex} for the families considered in \cite{BC}.  In fact, the $N/L$ constructed in \cite{BC} have the additional properties that they are Galois over $\mathbb{Q}$ and $\mathrm{Gal}(N/L)$ is central in $\mathrm{Gal}(N/\mathbb{Q})$. One can show that with these additional constraints, $m \le  \ell (\ell+ 1)/2$. The argument for this uses the growth of the successive quotients in the lower central $2$-series of a free pro-two group on $\ell$ generators.

\subsection{An infinite family of small discriminant fields} 
\label{s:inffamily}
We now summarize some consequences of the main result of \cite{BC}.  Let $\{p_1,\ldots,p_\ell\}$ be a set of distinct primes that are congruent to $1$ mod $4$.  Define $L = \mathbb{Q}(\sqrt{p_1},\ldots,\sqrt{p_\ell})$.  Let $\mathcal{S}_0$ be a finite set of elements of $O_L^* \cap (1 + 4 O_L)$ such that the images of the elements of $\mathcal{S}_0$ in the vector space $O_L^*/(O_L^*)^2$ over $\mathbb{Z}/2$ are linearly independent over $\mathbb{Z}/2$. 
\begin{enumerate}
\item[i.] Define $N$ to be the extension of $L$ generated by $\{\sqrt{w}: w \in \mathcal{S}_0\}$.  Then $N$ has degree $n(N) = 2^{\# \mathcal{S}_0 + \ell}$.  By Theorem \ref{ex:unitex}, 
$$\delta(N) = \delta(L) = (\prod_{i = 1}^\ell p_i)^{1/2}.$$
\item[ii.]  One can find an infinite family of fields $N$ of increasing degree $n(N)$ constructed as in (i) which the following is true. 
One has $\# \mathcal{S}_0 \ge c_0 \ell^2$ for some constant $c_0 > 0$ independent of $N$.
For each each $\epsilon > 0$, one has $\delta(N) < n(N)^\epsilon$ and $n(L) < n(N)^\epsilon$ for all but finitely many $N$ in the family.  
\medbreak
\item[iii.]  There is a polynomial $F(y) \in \mathbb{Q}[y]$ with the following property.  There is an infinite family of fields $N$ as in (ii) such that one can write down in time bounded by $F(\mathrm{log}(n(N)))$ the minimal polynomials over $\mathbb{Q}$ of the generators $\{\sqrt{p_i}: 1 \le i \le \ell\} \cup \{\sqrt{w}:w \in \mathcal{S}_0\}$ of $N$ over $\mathbb{Q}$.
 \end{enumerate}
 \medbreak
 
For $D \subset \{1,\ldots,\ell\}$ let 
\begin{equation}
\lambda_D = \prod_{i \in D} \frac{1 + \sqrt{p_i}}{2}
\end{equation}
Since $\mathbb{Q}(\sqrt{p_i})$ has ring of integers $\mathbb{Z} \oplus \mathbb{Z} \frac{1 + \sqrt{p_i}}{2}$ and the discriminants of the $\mathbb{Q}(\sqrt{p_i})$ as $p_i$ varies are disjoint, we have the following consequence of Theorem  \ref{ex:unitex}.
\begin{theorem}
\label{thm:nicebasis}  With the above notations
\begin{enumerate}
\item[i.] The set $\{\lambda_D: D \subset \{1,\ldots,\ell\} \}$ is a basis for $O_L$ over $\mathbb{Z}$. 
\item[ii.] For each subset $T$ of $\mathcal{S}_0$ let 
$$\eta_T = \prod_{w \in T} \frac{1+ \sqrt{w}}{2}$$
be as in Theorem \ref{ex:unitex}.  
The set $\{\eta_T :T \subset \mathcal{S}_0\}$ is a basis for $O_N$ as a free module for $O_L$.   
\item[iii.]The  set of products $\{\eta_T \cdot \lambda_D : T \subset \mathcal{S}_0, D \subset 
\{1,\ldots,\ell\} \}$ is a basis for $O_N$ over $\mathbb{Z}$.
\end{enumerate}
\end{theorem}

\section{Fundamental domains for finite index subgroups of  $O_N$.}
\label{s:fundament}

\subsection{Summary of the strategy to be used}
\label{s:fundstrategy}

Let  $\mathbb{Q} \subset L \subset N$ be  a two-step tower of number fields in which $N/L$ and $L/\mathbb{Q}$ are elementary abelian two-extensions.
We begin by describing a method for finding a fundamental domain for an arbitrary $\mathbb{Z}$-latttice $M$ in $\mathbb{R} \otimes_{\mathbb{Q}} N$.  

Let $\{q_i\}_{i = 1}^{[N:L]}$ be an ordered basis for $N$ over $L$.   Define $\{F_j\}_{j = 0}^{[N:L]}$ to be the filtration of $\mathbb{R} \otimes_{\mathbb{Q}} N$ defined by $$F_j = \bigoplus_{i \le j} \ (\mathbb{R} \otimes_{\mathbb{Q}} L) \cdot q_i.$$  Define $M_j = F_j \cap M$.    For $1 \le j \le [N:L]$ the quotient $M_j/M_{j-1}$ is a $\mathbb{Z}$-lattice in $F_j/F_{j-1}$ and  is identified with a $\mathbb{Z}$-lattice $M'_j$ in $(\mathbb{R} \otimes_{\mathbb{Q}} L )\cdot q_j$ via the natural isomorphism $$(\mathbb{R} \otimes_{\mathbb{Q}} L)\cdot q_j = F_j/F_{j-1}.$$ Let $V_j \subset (\mathbb{R} \otimes_{\mathbb{Q}} L)\cdot q_j$ be a fundamental domain for the translation action of $M'_j$ on $(\mathbb{R}\otimes_{\mathbb{Q}} L) \cdot q_j$. Then $$V = \bigoplus_{j =1}^{[N:L]} \ V_j \subset \bigoplus_{j = 1}^{[N:L]} \ (\mathbb{R} \otimes_{\mathbb{Q}} L) \cdot q_j = \mathbb{R} \otimes_{\mathbb{Q}} N $$ will be a fundamental domain for the translation action of $M$ on $\mathbb{R} \otimes_{\mathbb{Q}} N$.

For the rest of this section we assume that $N$ lies in an infinite family of fields  satisfying conditions (i), (ii) and (iii) of \S \ref{s:inffamily}.  In particular there is a set of distinct primes $\{p_1,\ldots,p_\ell\}$ congruent to $1$ mod $4$ with $L = \mathbb{Q}(\sqrt{p_1},\ldots,\sqrt{p_\ell})$, and $N = L(\sqrt{w}:w \in \mathcal{S}_0\}$ for a set of units $\mathcal{S}_0 \subset O_L^* \cap (1 + 4O_L)$.  For simplicity, we will also make the following hypothesis, which is satisfied by the families of $N$ constructed in \cite{BC}.

\begin{hypothesis}
\label{hyp:ortho} The set $\mathcal{S}_0$ is not empty, and for each $w \in \mathcal{S}_0$, the image of $w$ under every embedding $L \to \mathbb{R}$ is negative and
$L(\sqrt{w})$ is Galois over $\mathbb{Q}$. 
\end{hypothesis}

Let $\epsilon > 0$ be a parameter that we will let go to $0$ as $n \to \infty$.  Our goal is to describe a subgroup of finite index $O_{N,\epsilon}$ of $O_N$ together with a  fundamental domain $\mathcal{F}_{N,\epsilon}$ for the translation action of $O_{N,\epsilon}$ on $\mathbb{R} \otimes_{\mathbb{Q}} N$.

We first define an explicit basis $\{q_i\}_{i = 1}^{[N:L]}$ for $N$ over $L$ such that there is a 
basis $\{ \eta_i \}_{i = 1}^{[N:L]}$ for $O_N$ as a free $O_L$-module with the following property.  For all $j$, one has
$$\eta_j = 2^{-z(j)} q_j + \sum_{i < j} r_i q_i$$
for an integer $z(j) \ge 0$ and some $r_i \in L$.  If we set $M = O_N$ in the above discussion, we find $M'_j =  O_L 2^{-z(j)} q_j$.  However, it is not clear that there is a small fundamental domain for the translation action of $M'_j$ on $(\mathbb{R} \otimes_{\mathbb{Z}} L )\cdot 2^{-z(j)} q_j$.  So we use a Minkowski argument to construct, as a function of a positive constant $\epsilon > 0$ going to $0$, a non-zero element $h_j(\epsilon) \in O_L 2^{-z(j)}  q_j$ that has small $L^\infty$ and $L^2$ norm.  Since $O_L$ itself has a $\mathbb{Z}$-basis of small $L^\infty$ and $L^2$ norm, one can show that the sublattice $O_L h_j(\epsilon) \subset (\mathbb{R} \otimes_{\mathbb{Q}} L) \cdot 2^{-z(j)} q_j$ has a small fundamental domain $V_j(\epsilon)$ for its translation action on $(\mathbb{R} \otimes_{\mathbb{Q}} L)\cdot q_j$.   We now replace $M = O_N$ by the submodule 
$$O_{N,\epsilon} = \bigoplus_{j = 1}^{[N:L]} \ O_L \frac{h_j(\epsilon)}{2^{-z(j)} q_j} \eta_j.$$
Here $\frac{h_j(\epsilon)}{2^{-z(j)} q_j} \in O_L$ by construction, so $$O_{N,\epsilon} \subset O_N = \bigoplus_{j = 1}^{[N:L]} \ O_L \cdot \eta_j.$$  From the construction of a $\mathbb{Z}$-basis for
$O_L h_j(\epsilon) $ we have a procedure for moving elements of $(\mathbb{R} \otimes_{\mathbb{Z}} L) \cdot q_j$ into $V_j(\epsilon)$.  This leads to a procedure for moving elements of $\mathbb{R} \otimes_{\mathbb{Z}} N$ into 
$$\mathcal{F}_{N,\epsilon} = \bigoplus_{j =1}^{[N:L]} \ V_j(\epsilon)$$  by translation by elements of $O_{N,\epsilon}$.  
Theorems \ref{thm:tauthm1} and \ref{thm:tauinfthm1}  then follow on bounding the $L^2$ and $L^\infty$ radii of $\mathcal{F}_{N,\epsilon}$.  A useful fact in determining the $L^2$ radius of $\mathcal{F}_{N,\epsilon}$ is that the subspaces $(\mathbb{R} \otimes_{\mathbb{Q}} L) \cdot q_i$ are orthogonal with respect to the $L^2$ inner product on $\mathbb{R} \otimes_{\mathbb{Q}} N$.

\subsection{Orthogonal summands of  $\mathbb{R} \otimes_{\mathbb{Q}} N$.}
\label{s:innerproduct}
 Because of Hypothesis \ref{hyp:ortho}, $N$ is a  CM field, which means 
there is a non-trivial automorphism $\iota:N \to N$ of order $2$ such that for all embeddings $\sigma:N \to \mathbb{C}$ one has
$\overline{\sigma(\alpha)} = \sigma(\iota(\alpha))$ for all $\alpha \in N$.  Thus $N$ has no real embeddings, and the $n$ complex non-real embeddings occur in complex conjugate pairs $(\sigma_i , \sigma_i \circ \iota)$ for a set of embeddings $\sigma_1, \ldots, \sigma_{n/2}$. 

Let $\lambda:N \to \mathbb{C}^{n/2}$ be defined by $\lambda(\alpha) = (\sigma_i(\alpha))_{i = 1}^{n/2}$.  This extends to an isomorphism $\mathbb{R} \otimes_{\mathbb{Q}} N = \mathbb{C}^{n/2} $.  Define a real inner product on $\mathbb{C}^{n/2}$ by letting
$$\langle z, z' \rangle = \frac{1}{2} \sum_{i = 1}^{n/2} (z_i \overline{z'_i} + \overline{z_i} z'_i)$$
for  $z = (z_i)_{i = 1}^{n/2}$ and $z' =(z'_i)_{i = 1}^{n/2}$ in $\mathbb{C}^{n/2}$.
This inner product coincides with the usual Euclidean inner product on $\mathbb{R}^n$ when we identify $\mathbb{C}^{n/2}$ with $\mathbb{R}^n$
by identifying each $\mathbb{C}$ factor with $\mathbb{R}^2$ by taking real and imaginary parts.  In particular the $L^2$ norm on 
$\mathbb{R} \otimes_\mathbb{Q} N = \mathbb{C}^{n/2}$ is defined by $||z||_2 = \sqrt{\langle z , z \rangle}$.   The $L^\infty$ norm on 
$\mathbb{C}^{n/2}$ is simply the sup of the Euclidean norms of the coordinates of a vector.

\begin{lemma}
\label{lem:easynuff} If $\alpha, \beta \in N$ then the inner product $\langle \alpha ,\beta \rangle$ equals $\frac{1}{2}\mathrm{Trace}_{N/\mathbb{Q}}(\alpha \cdot \iota(\beta))$.
\end{lemma}

\begin{proof} By the definition of the trace,
\begin{eqnarray*}
\mathrm{Trace}_{N/\mathbb{Q}}(\alpha \cdot \iota(\beta)) &=& \sum_{i = 1}^{n/2} (  \sigma_i(\alpha \cdot \iota(\beta)) + (\sigma_i \circ \iota)(\alpha \cdot \iota(\beta)) )\\
&=& \sum_{i = 1}^{n/2} ( \sigma_i(\alpha) \overline{\sigma_i(\beta)} + \overline{\sigma_i(\alpha)} \sigma_i(\beta)) \\
& = &2 \langle \alpha,\beta \rangle.
\end{eqnarray*}
\end{proof}

\begin{lemma}
\label{lem:perp} When $T$ is a subset of $\mathcal{S}_0$ let $q_T = \prod_{w \in T} \sqrt{w}$.  Then the one-dimensional $L$-vector spaces $L \cdot q_T$ of $N$ as $T$ ranges over the subsets of $\mathcal{S}_0$ are mutually orthogonal.  Hence the same is true of the rank one $(\mathbb{R} \otimes_{\mathbb{Q}} L)$-submodules $(\mathbb{R}\otimes_{\mathbb{Q}} L) \cdot q_T$ of $\mathbb{R} \otimes_{\mathbb{Q}} N$.
\end{lemma}

\begin{proof}  We have $q_T^2 \in L$, and $L$ is a totally real field, so $\iota$ fixes $L$ and $\iota(q_T) = \pm q_T$. 
Suppose $s,s' \in L$.  Suppose $T$ and $T'$ are distinct subsets of $\mathcal{S}_0$. Then 
\begin{equation}
\label{eq:TTprime}
2 \langle s q_T, s' q_{T'} \rangle  = \mathrm{Trace}_{N/\mathbb{Q}} (s q_T \cdot \iota(s' q_{T'})) = \mathrm{Trace}_{L/\mathbb{Q}} (s s' \ \mathrm{Trace}_{N/L}(q_T \cdot \iota (q_{T'})).
\end{equation}
Here $q_T \cdot \iota(q_{T'}) = \pm q_T \cdot q_{T'} $ is not in $L$ but its square is in $L$ since $T$ and $T'$ are distinct.  Hence
$q_T \cdot \iota(q_{T'})$ generates a quadratic extension $L'$ of $L$ in $N$, and $\mathrm{Trace}_{L'/L}(q_T \cdot \iota(q_{T'}))= 0$.  It follows that $\mathrm{Trace}_{N/L}(q_T \cdot \iota(q_{T'})) = 0$, and the lemma now follows from (\ref{eq:TTprime}).
\end{proof}

\subsection{Constructing the subgroups  $O_{N,\epsilon} \subset O_N$ and the fundamental domains $\mathcal{F}_{N,\epsilon}$.}

Recall that $n = n(N) = 2^{\# \mathcal{S}_0 + \ell} $ is the degree of $N$.

\begin{lemma}
\label{lem:keypoint}  Suppose $\epsilon > 0$.  Let $T$ be a subset of $\mathcal{S}_0$.  For sufficiently large $n$, there is a basis $B_{T,\epsilon}$ over $\mathbb{Z}$ for a finite index subgroup $M_{T,\epsilon}$ of  $O_{L} \cdot q_T$ such that
\begin{equation}
\label{eq:bbounds}
|| b ||_\infty  \le n^{\epsilon} \quad \mathrm{and} \quad ||b||_2 < n^{1/2 + \epsilon}
\end{equation}
for all $b \in B_{T,\epsilon}$ and 
\begin{equation}
\label{eq:indexbound}
[O_L \cdot q_T : M_{T,\epsilon}] \le  
n^{3 \, \epsilon \,n(L)/2}.
 \end{equation}
 \end{lemma}

\begin{proof}  By Theorem 1.3 of \cite{BC}, we can assume there is a constant $X$ going to $\infty$ as $n \to \infty$ for which the following is true.  The primes $p_i$ are bounded above by $X$, and the degree $n = n(N)$ of $N$ satisfies 
\begin{equation}
\label{eq:degbounder}
\mathrm{log}(n) \ge c X^2/(\mathrm{log}(X))^2
\end{equation} for some universal positive constant $c$. 

From Theorem \ref{thm:nicebasis}, a  basis $B_L$ over $\mathbb{Z}$ for $O_{L}$ is given by the products
$$\lambda_D = \prod_{i \in D} \frac{1 + \sqrt{p_i}}{2}$$
as $D$ ranges over subsets of $\{1,\ldots,\ell \}$.  Since the $p_i$ are bounded by $X$, all the (real) conjugates of these products are bounded in absolute value by $\sqrt{X!}$
where $\mathrm{log} \sqrt{X!} < X \mathrm{log}(X)$.   Hence all  $m \in B_L$  have $\mathrm{log}(||m||_\infty) \le X \mathrm{log}(X)$.  From (\ref{eq:degbounder}) we have 
$$\frac{\mathrm{log}(||m||_\infty)}{\mathrm{log}(n)} \le \frac{(\mathrm{log}(X))^3}{c X}.$$
 So if $X$ is sufficiently large, 
 \begin{equation}
 \label{eq:mboundit}
 ||m||_\infty \le n^{\epsilon/2}\quad \mathrm{for}\quad m \in B_L.
 \end{equation}
   Since there are $n/2$ complex conjugate pairs of embeddings of $N$ into $\mathbb{C}$, this leads to 
 $||m||_2 \le 2^{-1/2} n^{1/2 + \epsilon/2} < n^{1/2+\epsilon}$ if $m \in B_L$.  
 
 When $T = \emptyset$ then $q_T = 1$ and the statement of Lemma \ref{lem:keypoint} holds when $B_{T,\epsilon} = B_L$ and $M_{T,\epsilon} = O_L$.  From now on we suppose that $T$ is a non-empty subset of $\mathcal{S}_0$.

 For each embedding $\sigma:L \to \mathbb{C}$ choose an extension of $\sigma$ to an embedding $\sigma:N \to \mathbb{C}$.  Let $\Psi_L$ be the set of these extensions, so $\# \Psi_L = n(L)$.

We have 
$q_T^2 = \prod_{w\in T} w \in O_L^*$.  For $\sigma \in \Psi_L$, let $e_{\sigma}(T) = 0$ if $\sigma(q_T^2) > 0$, and let $e_\sigma(T) = 1$ otherwise.  We have a diagram of vector space isomorphisms

\begin{equation}
\label{eq:origin}
 \xymatrix  {
 \mathbb{R} \otimes_{\mathbb{Q}} L \ar[d]^{\cdot q_T}\ar[rr] &&\prod_{\sigma \in \Psi_L} \mathbb{R}\ar[d]^{\cdot \prod_{\sigma \in \Psi_L} \sigma(q_T)}\\
(\mathbb{R} \otimes_{\mathbb{Q}} L) q_T \ar[rr]&& \prod_{\sigma \in \Psi_L} \mathbb{R}\cdot (\sqrt{-1})^{e_\sigma(T)}
}
\end{equation}
in which the horizontal homomorphisms are induced by the $\sigma \in \Psi_L$, the left vertical map is induced by multiplication by $q_T$, and the right vertical map on the component indexed by $\sigma $ is multiplication by $\sigma(q_T)$.   We give the terms on the right side the volume form induced by the usual Euclidean distance function on each of the factors.  The right vertical map then multiplies volume by 
$$\prod_{\sigma \in \Psi_L} |\sigma(q_T)| = |\prod_{\sigma \in \Psi_L} \sigma(q_T^2) |^{1/2} = |\mathrm{Norm}_{L/\mathbb{Q}} (q_T^2)|^{1/2}$$
since  the restriction of the elements of $\Psi_L$ to $L$ give the elements of $\mathrm{Gal}(L/\mathbb{Q})$.  However, $q_T^2 \in O_L^*$ implies
$\mathrm{Norm}_{L/\mathbb{Q}}(q_T^2) = \pm 1$.  Hence the right vertical map preserves volumes.  This map sends the image of $O_L$ to the image of $O_L q_T$.  We conclude that
\begin{equation}
\label{eq:delvol}
\Delta(L)^{1/2} = \mathrm{Vol}( (\mathbb{R}\otimes_{\mathbb{Q}} L)/ O_L) = \mathrm{Vol}((\mathbb{R} \otimes_{\mathbb{Q}} L)q_T/ O_L q_T).
\end{equation}

Suppose $\alpha \in \prod_{\sigma \in \Psi_L} \mathbb{R}\cdot (\sqrt{-1})^{e_\sigma(T)} = (\mathbb{R} \otimes_{\mathbb{Q}} L) q_T$.  The sup of the Euclidean absolute values of the coordinates of $\alpha$ is then equal to the sup $||\alpha||_\infty$ of the Euclidean norms of the images of $\alpha$ under  the $\mathbb{R}$-linear extensions of all the embeddings of $N$ into $\mathbb{C}$.
We now apply Minkowski's theorem  to the lattice
 $$O_{L} q_T \subset (\mathbb{R} \otimes_{\mathbb{Q}}L )q_T =   \prod_{\sigma \in \Psi_L} \mathbb{R}\cdot (\sqrt{-1})^{e_\sigma(T)}.$$
This and the volume 
computation (\ref{eq:delvol}) show that there is a non-zero element $h_T$ of $ O_Lq_T$ such that 
$$ ||h_T||_\infty \le  \Delta(L)^{1/(2n(L))}.$$  

 Here 
 $$\Delta(L)^{1/(2 n(L))}  = 
\Delta(N)^{1/(2n(N))} \le n^{\epsilon /2}$$
 for all large $n$ by Theorem 1.3 of \cite{BC}.  Thus
 \begin{equation}
 \label{eq:hinfinity}
 ||h_T||_\infty  \le n^{\epsilon/2}.
 \end{equation}
We now take the basis $B_{T,\epsilon}$ in the lemma to be $\{h_T \cdot m: m \in B_L\}$.
 
Suppose $m \in B_L$, so that we have already shown in (\ref{eq:mboundit}) that  $||m||_\infty \le n^{\epsilon / 2}$. We then have
  \begin{equation}
 \label{eq:hfind2}
  ||h_T \cdot m||_\infty \le n^{\epsilon/2} \cdot n^{\epsilon/2} = n^\epsilon.
  \end{equation} 
  This gives  
 $$||h_T \cdot m||_2^2 \le (n/2) \cdot n^{2 \epsilon} \le n^{1 + 2\epsilon}$$ since $N$ has $n$ complex embeddings.  This proves (\ref{eq:bbounds}).
 
 For $\alpha \in (\mathbb{R} \otimes_{\mathbb{Q}} L) q_T = \prod_{\sigma \in \Psi_L} \mathbb{R}\cdot (\sqrt{-1})^{e_\sigma(T)}$ let $||\alpha||'_2$ be the square root of the sum of the squares of the coordinates of $\alpha$.  Thus $|| \ ||'_2$ is just the usual Euclidean norm on $(\mathbb{R} \otimes_{\mathbb{Q}} L)q_T$ as a real vector space of dimension $n(L)$.
 Since $L$ is totally real with $n(L)$ real embeddings we find from (\ref{eq:hfind2}) that  
  \begin{equation}
  \label{eq:lambound}
  ||h_T\cdot m||'_2 \le \sqrt{n(L)}\  n^\epsilon \quad \mathrm{for} \quad m \in B_L.
  \end{equation}
  Here $\sqrt{n(L)} < n^{\epsilon/2}$ when $n$ is sufficiently large by (ii) of \S \ref{s:inffamily}.
 Since $\{h_T \cdot m: m \in B_L\} = B_{T,\epsilon}$ is a basis for $M_{T,\epsilon}$ over $\mathbb{Z}$ we find from this and (\ref{eq:lambound}) that 
 $$\mathrm{Vol}((\mathbb{R} \otimes_{\mathbb{Q}} L) q_T / M_{T,\epsilon}) \le (\sqrt{n(L)} \ n^\epsilon)^{n(L)} \le  n^{3 \, \epsilon \, n(L)/2}.$$
The index $[O_L q_T:M_{T,\epsilon}]$ is the ratio of the covolumes of $M_{T,\epsilon}$ and $O_L q_T$.  Since $O_L q_T$ has covolume  $\Delta(L)^{1/2} \ge 1$ by
(\ref{eq:delvol}), we conclude
$[M_{T,\epsilon}:O_L q_T] \le n^{3 \, \epsilon\,  n(L)/2}$. This completes the proof of Lemma \ref{lem:keypoint}.
   \end{proof}
 
 \begin{remark}
 \label{rem:compute}  If $T$ is empty $B_{T,\epsilon}$ can be taken to be the $\mathbb{Z}$-basis $B_L$ of $O_L$ described in Theorem \ref{thm:nicebasis}.   Suppose now that $T$ is not empty.  One can find an  element $h_T \in O_L q_T$ for which (\ref{eq:hinfinity}) holds by solving a short vector problem for a lattice of rank $n(L) = 2^\ell$.   Note that $n(L)$ is much smaller than $n(N)$.  \end{remark}

 Recall from Theorem \ref{thm:nicebasis} that $O_N$ has a basis as a free $O_{L}$-module consisting of the products
 $$\eta_T = \prod_{w \in T} \frac{1 + \sqrt{w}}{2}$$
 as $T$ ranges over all subsets of $\mathcal{S}_0$.  Expanding the right hand side of this product shows $\eta_T$ is a $\mathbb{Q}$-linear combination of the products
 $$q_{T'} = \prod_{w \in T'} \sqrt{w}$$
 as $T'$ ranges over subsets of $T$.  For example, when $T' = T$, we get the term $2^{-\# T} q_T$ on expanding the product for $\eta_T$.

 \begin{definition}
 \label{dfn:funda}  Suppose $\epsilon > 0$ and $T \subset \mathcal{S}_0$.  Let $\mathcal{F}_{T,\epsilon}$ be the subset of $\mathbb{R} \otimes_{\mathbb{Q}} N$ consisting of all vectors 
 $$\sum_{b \in B_{T,\epsilon}} r_b b$$
 in which the real numbers $r_b$ satisfy 
 \begin{equation}
 \label{eq:rbound}
 0 \le r_b < \frac{1}{2^{\# T}}.
 \end{equation}
 Define $\mathcal{F}_{N,\epsilon}$ to be the direct sum of the $\mathcal{F}_{T,\epsilon}$ as $T$ ranges over all subsets of $\mathcal{S}_0$.
 Let $O_{N,\epsilon}$ be the subgroup of $O_N$ generated by the elements $(b q_T^{-1}) \eta_T$ as $T$ ranges over all subsets of $\mathcal{S}_0$ and $b$ ranges over $B_{T,\epsilon}$. 
 \end{definition}
 
 Note that $O_{N,\epsilon}$ has finite index in $O_N$ since the $b \in B_{T,\epsilon}$ give a basis over $\mathbb{Z}$ for a finite index subgroup $M_{T,\epsilon}$ of $O_L \cdot q_T \subset O_N$ and the $\eta_T$ give a basis for $O_N$ over $O_L$ as $T$ ranges over subsets of $\mathcal{S}_0$.
 
 \begin{theorem}
 \label{thm:mainthm1}  Suppose $\epsilon > 0$.  Then $\mathcal{F}_{N,\epsilon}$ is a fundamental domain for the translation action of $O_{N,\epsilon}$ on $\mathbb{R} \otimes_{\mathbb{Q}} N$.  There is a procedure for moving an arbitrary element of $\mathbb{R} \otimes_{\mathbb{Q}} N$ into $\mathcal{F}_{N,\epsilon}$ by translation by an element of $O_{N,\epsilon}$.  
  \end{theorem}
 
 \begin{proof}
Extend the inclusion partial ordering of the subsets $T$ of $\mathcal{S}_0$ to a total ordering. We expand elements $\alpha\in \mathbb{R}\otimes_{\mathbb{Q}} N$ in a unique way as  
\begin{equation}
\label{eq:expansionalpha}
\alpha = \sum_{T \subset \mathcal{S}_0} \alpha_T
\end{equation}
where  $\alpha_T  \in (\mathbb{R} \otimes_{\mathbb{Q}} L) \cdot q_T$.    Suppose that  $T = T(\alpha)$ is the maximal element of $\mathcal{S}_0$ with respect to the above total ordering that occurs non-trivially in the expansion of $\alpha$. 
We will then call $\alpha_T$ the highest degree term of $\alpha$.  We can write
$$ \alpha_T = \sum_{b \in B_{T,\epsilon}} s_b 2^{-\# T} b$$
for some unique real numbers $s_b$.  
For each $b \in B_{T,\epsilon}$, let $\lfloor s_b \rfloor$ be the largest integer less than or equal to $s_b$.
Then 
\begin{equation}
\label{eq:nicit}
\alpha_T =  \sum_{b \in B_{T,\epsilon}} \lfloor s_b \rfloor\cdot  2^{-\# T} \cdot b + 
\sum_{b \in B_{T,\epsilon}} r_b \cdot  b
\end{equation}
where $0 \le r_b  = 2^{-\# T} (s_b - \lfloor s_b \rfloor) < 2^{-\#T}$.  The sum
\begin{equation}
\label{eq:taudef}
\tau = \sum_{b \in B_{T,\epsilon}} \lfloor s_b \rfloor \cdot b \cdot q_T^{-1} \cdot \eta_T = \sum_{b \in B_{T,\epsilon}} \lfloor s_b \rfloor \cdot b \cdot q_T^{-1} \cdot \left( \sum_{T' \subset T} 2^{-\# T} q_{T'} \right ) 
\end{equation} lies
in $O_{N,\epsilon}$ and has highest degree term
 $$ \sum_{b \in B_{T,\epsilon}} \lfloor s_b \rfloor\cdot 2^{-\# T} b.$$

Consider the expansion of  
$$\beta = \alpha  - \tau$$
as an $(\mathbb{R} \otimes_{\mathbb{Q}} L)$-linear combination of the $q_{T'}$ as $T'$ ranges over subsets of $\mathcal{S}_0$.  In $\tau$, only the $q_{T'}
$ associated to $T' \subset T$ occur.  So subtracting $\tau$ from $\alpha$ affects only the coefficients of these $q_{T'}$.  The highest degree term in $\beta$ is
$$\sum_{b \in B_{T,\epsilon}} r_b \cdot  b$$
and this lies in the summand of $\mathcal{F}_{T,\epsilon}$ of $\mathcal{F}_{N,\epsilon}$.  Furthermore, $\tau$ is the unique $\mathbb{Z}$-linear combination of elements of $\{b q_T^{-1} \eta_T: b \in B_{T,\epsilon}\}$  for which the resulting $\beta$ will have its highest degree term in $\mathcal{F}_{T,\epsilon}$, since $B_{T,\epsilon}$ is a basis for
$(\mathbb{R}\otimes_{\mathbb{Q}} L) \cdot q_T$ over  $\mathbb{R}$.  
By induction, we can continue this process in order to find explicitly a unique element of $O_{N,\epsilon}$ that brings $\alpha$ into $\mathcal{F}_{N,\epsilon}$.  \end{proof} 

The fundamental domain $\mathcal{F}_{N,\epsilon}$ we have constructed for $O_{N,\epsilon}$ certainly contains a fundamental domain for $O_N$, but it is a natural question how much larger the former is than the latter. We will find a bound for the volume of the former in terms of the volume of the latter by finding an upper bound on the index of $O_{N,\epsilon}$ in $O_N$. 
\begin{theorem}
\label{thm:ohyeah}  For all $\epsilon > 0$, if $n = n(N)$ is sufficiently large, then 
$$[O_N:O_{N,\epsilon}] \le n^{3\, n\, \epsilon/2}.$$
\end{theorem}

\begin{proof}  The group 
$O_{N,\epsilon}$ (resp. $O_{N}$) is the direct sum of the groups  $M_{T,\epsilon} q_T^{-1} \eta_T$ (resp. $O_L \eta_T$) as $T$ ranges over subsets of $\mathcal{S}_0$.   Hence (\ref{eq:indexbound}) gives 
\begin{equation}
\label{eq:niceenuff}
[O_N:O_{N,\epsilon}]= \prod_{T \subset \mathcal{S}_0} [O_L q_T:M_{T,\epsilon}] \le (n^{3 \, \epsilon/2 \, n(L)})^{ 2^{\# \mathcal{S}_0}} = n^{3\, n\, \epsilon/2}
\end{equation}
since $n = n(N) = n(L) \cdot 2^{\# \mathcal{S}_0}$.  
\end{proof}

Here is a consequence of this result.  The covolume of $O_N$ in $\mathbb{R} \otimes_{\mathbb{Q}} N$ is
$$2^{-n/2} \sqrt{\Delta(N)} \le n^{\epsilon \, n/2}.$$
Theorem \ref{thm:ohyeah} shows that the covolume of $O_{N,\epsilon}$ in $\mathbb{R} \otimes_{\mathbb{Q}} N$ is bounded by 
$$n^{\epsilon \, n/2} \cdot n^{3\, n\, \epsilon/2} = n^{2\, \epsilon\, n}.$$

 \subsection{Proof of Theorems \ref{thm:tauthm1} and  \ref{thm:tauinfthm1}}
 \label{s:boundfundamental}
 
\begin{definition} The $L^2$ radius $||\mathcal{F}_{N,\epsilon}||_2$ of $\mathcal{F}_{N,\epsilon}$ is defined to be the supremum of $||\alpha ||_2$ as $\alpha$  ranges over $\mathcal{F}_{N,\epsilon}$. Define the 
$L^\infty$ radius $||\mathcal{F}_{N,\epsilon}||_\infty$  similarly.
\end{definition}

\begin{lemma}
\label{lem:upperb} For all $\epsilon > 0$, if $n = n(N)$ is sufficiently large, one has 
\begin{equation}
\label{eq:upboundit}
||\mathcal{F}_{N,\epsilon}||_\infty \le n^{\mathrm{log}_2(3/2) + 2\epsilon} \quad \mathrm{and} \quad 
 ||\mathcal{F}_{N,\epsilon}||_2 \le  n^{\frac{1}{2}+ \mathrm{log}_2(\sqrt{5}/2) + 2\epsilon}.
\end{equation} The upper bounds in Theorems  \ref{thm:tauthm1}  and  \ref{thm:tauinfthm1} hold.
\end{lemma}
\begin{proof} In the following arguments, $n$ is assumed to be sufficiently large.  By Lemma \ref{lem:keypoint},
$|| b ||_\infty  \le n^{\epsilon}$ for all $b \in B_{T,\epsilon}$ and all $T \subset \mathcal{S}_0$.  In the definition of $\mathcal{F}_{N,\epsilon}$ in Definition \ref{dfn:funda}, the coefficient of each such $b$ is bounded by $2^{-\# T}$. Here $\# B_{T,\epsilon} = n(L) = 2^\ell \le n^\epsilon$.  The number of $T$ with $\# T = j$ is the binomial coefficient $\left ( {\# \mathcal{S}_0}\atop {j} \right )$.  The triangle inequality and the binomial theorem now show that if $\alpha \in \mathcal{F}_{N,\epsilon}$ then
\begin{equation}
\label{eq:alphainf}
||\alpha||_\infty  \le  \sum_{T \subset \mathcal{S}_0} 
2^{-\# T} n^{2 \epsilon} \le (1 + 1/2)^{\# \mathcal{S}_0} n^{2 \epsilon} 
= (3/2)^{\# \mathcal{S}_0} n^{2 \epsilon}.
\end{equation}
Since $n = 2^{\# \mathcal{S}_0 + \ell} $ we have
$$\mathrm{log}_2((3/2)^{\# \mathcal{S}_0})  = \#\mathcal{S}_0 \cdot \mathrm{log}_2(3/2) \le \mathrm{log}_2(n)\cdot  \mathrm{log}_2(3/2).$$
so $(3/2)^{\# \mathcal{S}_0} \le n^{\mathrm{log}_2(3/2)}$ and we have  the first inequality in (\ref{eq:upboundit})  from 
(\ref{eq:alphainf}).

One shows  the second inequality of (\ref{eq:upboundit}) in a similar way, but one can take advantage of the fact that 
$ L q_T$ is perpendicular to $L q_{T'}$ if $T \ne T'$.  We know
from Lemma \ref{lem:keypoint} that $||b||_2 \le n^{1/2 + \epsilon}$ for all $b \in B_{T,\epsilon}$ and all $T$. As above, $\# B_{T,\epsilon} \le n^\epsilon$ for all $T$.  
So if $\alpha \in \mathcal{F}_{N,\epsilon}$ then 
$$(||\alpha||_2)^2 \le  \sum_{T \subset \mathcal{S}_0}  2^{-2\# T} n^{1 + 4\epsilon} \le (1 + 1/4)^{\# \mathcal{S}_0} n^{1 + 4\epsilon} 
=  (5/4)^{\# \mathcal{S}_0}n^{1 + 4\epsilon} .$$
This and $n = 2^{\# \mathcal{S}_0 + \ell} $ give the second inequality in (\ref{eq:upboundit}).

We now have the upper bounds in Theorems \ref{thm:tauthm1} and  \ref{thm:tauinfthm1} by letting $\epsilon$ go to $0$, noting that for each $\epsilon > 0$ the above arguments apply for sufficiently large $n$.  
\end{proof}

The lower bound in Theorem \ref{thm:tauinfthm1} is trivial by a volume argument.   
The following result establishes the lower bound in Theorem \ref{thm:tauthm1}. 

\begin{lemma}
\label{lem:arithgeom} Let $F$ be an arbitrary number field and suppose $0 \ne \alpha \in O_F$. Then
the $L^2$ norm $||\alpha||_2 = ||\alpha||_{2,F}$ of $\alpha$  satisfies
\begin{equation}
\label{eq:theL2bound}
||\alpha||_2 \ge \sqrt{n(F)/2}.
\end{equation}
Thus the $L^2$ ball in $\mathbb{R} \otimes_{\mathbb{Q}} F$ of any radius less than  $\sqrt{n(F)/2}$ 
intersects $O_F$ in $\{0\}$.

\end{lemma}

\begin{proof}  First suppose that $F$ is totally complex.  We apply the arithmetic geometric mean inequality to the set $\{|\sigma(\alpha)|^2\}_\sigma$ where $\sigma$ ranges over all embeddings $\sigma:F \to \mathbb{C}$.  This leads to  
$$1 \le \left(\mathrm{Norm}_{F/\mathbb{Q}}( \alpha)\right) ^{2/(n(F))} = \left (\prod_{\sigma} |\sigma(\alpha)|^2 \right )^{1/(n(F))} \le \frac{1}{n(F)} \sum_{\sigma} |\sigma(\alpha)|^2 = \frac{2}{n(F)}  || \alpha||_2^2$$
which shows (\ref{eq:theL2bound}) in the totally complex case.  Now suppose that $F$ is not totally complex.  Then $F' = F(\sqrt{-1})$ is a degree $2$ totally complex extension of $F$.  Let $|| \ ||_{2,F'}$ be the $L^2$ norm of $\mathbb{R} \otimes_{\mathbb{Q}} F'$.  Each real place of $F$ extends to a unique complex place of $F'$, while each complex place of $F$ extends to two complex places of $F'$.  So for $0 \ne \alpha \in O_F$ we have
$$||\alpha||^2_{2,F} \le ||\alpha||^2_{2,F'} \le 2||\alpha||^2_{2,F}.$$
This and (\ref{eq:theL2bound}) in the totally complex case show
$$||\alpha||_{2,F} \ge \frac{1}{\sqrt{2}} ||\alpha||_{2,F'} \ge \frac{1}{\sqrt{2}} \cdot \sqrt{n(F')/2} = \sqrt{n(F)/2}.$$
\end{proof}

\section{Lower bounds for $L^2$ radii in terms of discriminants}
\label{s:cyclocase}

The following result explains the comment after Question \ref{q:firstquest1} concerning lower bounds for the $L^2$ radius of the Voronoi cell of a number field in terms of the root discriminant of the field.

\begin{theorem}
\label{thm:volbound} Let   $||V_2(K)||_2$ be the $L^2$ radius of the $L^2$ Voronoi cell $V_2(K)$ of the integers $O_K$ of a number field $K$.  Write $||V_2(K)||_2 = d\sqrt{n(K)}$ for some real number $d$.  Then as $n(K) \to + \infty$ we have 
\begin{equation}
\label{eq:thebound}
\frac{1}{2}  \mathrm{log}(\delta(K)) - \frac{r_2(K)}{n(K)} \mathrm{log}(2) - \frac{1}{2}(1 + \mathrm{log}(2\pi)) \le \mathrm{log}(d) + O\left(\frac{\mathrm{log}(n(K))}{n(K)}\right)
\end{equation}
where $r_2(K) \le n(K)/2$ is the number of complex places of $K$.
\end{theorem}

\begin{proof}     The volume of $V_2(K)$ is
$2^{-r_2(K)} \Delta(K)^{1/2}$. For simplicity let $n = n(K)$ be the degree of $K$, and write $||V_2(K)||_2 = d\sqrt{n}$ for some $d \in \mathbb{R}$.  Then $V_2(K)$ has volume
bounded above by the volume of an $n$-dimensional ball of radius $d \sqrt{n}$.  This leads to 
\begin{equation}
\label{eq:vollower}
2^{-r_2(K)} \Delta(K)^{1/2} \le \left (d\sqrt{n}\right )^{n} \frac{ \pi^{n/2}}{\Gamma(\frac{n}{2} + 1)}.
\end{equation}
 Stirling's formula says 
 $$\mathrm{log}(\Gamma(z)) = z \cdot \mathrm{log}(z) - z + O(\mathrm{log}(z))$$
  as $z \to +\infty$.  Taking logarithms of both sides of (\ref{eq:vollower}) and applying Stirling's formula shows
 \begin{equation}
 \label{eq:lower2}
 -r_2(K) \mathrm{log}(2) + \frac{1}{2} \mathrm{log}(\Delta(K)) \le n\left(\mathrm{log}(d) + \frac{1}{2}(1 + \mathrm{log}(2 \pi))\right) + O(\mathrm{log}(n)).
 \end{equation}
 Rearranging (\ref{eq:lower2})  shows (\ref{eq:thebound}).
 \end{proof}

For lack of a suitable reference, we  conclude with a proof of the statement made after Theorem \ref{thm:tauthm1}  that cyclotomic fields are not sufficient to prove an upper bound on $\nu_2$ in Theorem \ref{thm:tauthm1} that is smaller than $1$.

 \begin{corollary}
 \label{cor:cyclocor}
 Suppose $\epsilon > 0$. For all but finitely many cyclotomic fields $K$ one has
\begin{equation}
\label{eq:lowerboundnice} \delta(K) \ge n(K)^{1-\epsilon}\quad \mathrm{and}\quad ||V_2(K)||_2 \ge n(K)^{1 - \epsilon}.
\end{equation}
\end{corollary}

\begin{proof}
The second inequality in (\ref{eq:lowerboundnice}) follows from the first one and Theorem \ref{thm:volbound}. We now prove the first inequality in (\ref{eq:lowerboundnice}). Suppose $K = \mathbb{Q}(\zeta_m)$
for a primitive $m^{th}$ root of unity $\zeta_m$. Write $m = \prod_i p_i^{r_i}$ for some distinct primes $p_i$ and integers $r_i \ge 1$.  The classical formula for the discriminant of $K$ along with $n(K) = \phi(m) = \prod_i p^{r_i-1} (p_i - 1)$ give
$$\mathrm{log}(\delta(K)) = \frac{\mathrm{log}(\Delta(K))}{n(K)} = \sum_i \left(r_i - \frac{1}{p_i - 1}\right) \mathrm{log}(p_i).$$
Since $n(K) = \phi(m) < m$ we find that
\begin{equation}
\label{eq:almostdonefin}
\mathrm{log}(\delta(K)) - (1 - \epsilon)\cdot \mathrm{log}(n(K)) > \mathrm{log}(\delta(K))- (1 - \epsilon)\cdot \mathrm{log}(m)
= \sum_i e(p_i,r_i)
\end{equation}
where
$$e(p_i,r_i) = \left(\epsilon r_i - \frac{1}{p_i - 1}\right)\cdot \mathrm{log}(p_i).$$
For a given $\epsilon > 0$ and a real number $B$, let $M(B)$ be the set of all pairs $(p,r)$ consisting of a prime $p$ and an integer $r \ge 1$ such
that $e(p,r) \le B$.  Then $M(B)$ is finite.  Let $C$ be the sum of $e(p,r)$ over all $(p,r) \in M(0)$.  If $m$ is sufficiently large, some $(p_i,r_i)$ arising from the factorization of $m$ must lie outside of the finite set $M(-C)$.  Then the right hand side of (\ref{eq:almostdonefin}) is positive, which completes the proof.
\end{proof}

\bibliographystyle{plain}

 \end{document}